\documentclass[10 pt]{amsart}
\usepackage{amsfonts}
\usepackage{ifthen}
\usepackage{amsthm}
\usepackage{amsmath}
\usepackage{graphicx}
\usepackage{amscd,amssymb,amsthm}
\usepackage{graphicx}
\usepackage{epstopdf}
\usepackage{hyperref}
\usepackage{cleveref}
\usepackage{cite}

\newcounter{minutes}
\divide\time by 60
\newcounter{hours}
\multiply\time by 60 \addtocounter{minutes}{-\time}

\newtheorem{lemma}{Lemma}[section]
\newtheorem{theorem}{Theorem}[section]
\newtheorem{corollary}{Corollary}[section]

\newcommand{\real}{\operatorname{Re}}

\keywords{Generalized Struve functions; univalent, starlike and convex functions; radius of starlikeness and convexity; Mittag-Leffler expansions; Laguerre-P\'olya class of entire functions.}
\subjclass[2010]{30C45, 30C15, 33C10}

\begin{document}

\title{Radii of starlikeness and convexity of generalized Struve functions }

\author[E. Toklu]{Evr{\.I}m Toklu}
\address{Department of Mathematics, Faculty of Education, A\u{g}r{\i} {\.I}brah{\i}m \c{C}e\c{c}en University, 04100 A\u{g}r{\i}, Turkey} 
\email{evrimtoklu@gmail.com}

\def\thefootnote{}
\footnotetext{ \texttt{File:~\jobname .tex,
printed: \number\year-0\number\month-\number\day,
\thehours.\ifnum\theminutes<10{0}\fi\theminutes}
} \makeatletter\def\thefootnote{\@arabic\c@footnote}\makeatother

\maketitle

\begin{abstract}
In this paper, it is aimed to determine the radii of starlikeness and convexity of the normalized generalized Struve functions for three different kinds of normalization and to find tight lower and upper bounds for the radius of starlikeness and convexity of these normalized Struve functions by making use of Euler-Rayleigh inequalities. The Laguerre-P\'olya class of entire functions has a crucial role in constructing our main results.
\end{abstract}

\section{\bf Introduction and Prerequisites}
As it is well known, special functions are one of the most powerful tools in the solution of a wide variety of important problems. Because of the fruitful properties of special functions, it is important to examine their properties in many aspects. In the recent years, there has been a vivid interest on geometric properties of special functions from the point of view of geometric function theory, like Bessel, Struve and Lommel functions of the first kind; see the papers \cite{AB,ABO,ABY,ATO,BDOY,BKS,BOS,BSz1,BSz2,BTK} and the references therein. However, it is possible to say that the roots of special functions seen on geometric functions theory are based on the studies of Brown \cite{Brown}, Kreyszig and Todd \cite{KT } and Wilf \cite{Wilf} which initiated investigation on the univalence of Bessel functions and on the determination the radius of starlikeness for different kinds of normalization. Recently, in 2014, Baricz \textit{et al.} \cite{BKS} came up with a much simpler approach for determining the radius of starlikeness of normalized Bessel functions of first kind. In the same year, Baricz and Sz\'{a}sz \cite{BSz2} obtained the radius of convexity of the same normalized Bessel functions. When scrutinized their studies (in particular,  \cite{BKS} and \cite{BSz2}) the main facts pertaining to these studies can be stated as follows: the radii of univalence, starlikeness and convexity are obtained as solutions of some transcendental equations and the obtained radii satisfy some interesting inequalities. In addition, the other main fact is that the positive zeros of Bessel, Struve, Lommel functions of the first kind and the Laguerre-P\'olya class of entire functions have an important place in these papers. Recent years, there has been extensive work on determining some geometric properties of other special functions involving Bessel function of first kind such as univalence, starlikeness, convexity and so forth. For instance, in \cite{DSz}, Deniz and Sz\'asz obtained the radius of uniform convexity of the normalized Bessel functions. And also, Bohra and Ravichandran in \cite{BR} investigated the radius of strong starlikeness and \(k\) uniform convexity of order \(\alpha\) of the normalized Bessel functions.

Motivated by the above series of papers on geometric properties of special functions and using the method of Baricz \textit{et al.} \cite{BKS}, our aim in this paper is to present some similar results for the normalized forms of the generalized Struve functions, which is similar to the generalized Bessel functions of Wright type (and of Galue type).

This paper is organized as follows: The rest of this section contains some basic definitions needed for the proof of our main results. Section 2 is divided into three subsections: The first subsection is devoted to the study on the radii of starlikeness of normalized generalized Struve functions. The second subsection contains the study of the radii of convexity of  normalized generalized Struve functions. The last subsection is allocated to the presentation of making some comparisons our main results given in this paper with obtained results earlier. 

Before starting to present our main results, we would like to draw attention to some basic concepts needed for building our main results. For $r>0$ we denote by $\mathbb{D}_r=\left\{z\in\mathbb{C}: |z|<r\right\}$ the open disk of radius $r$
centered at the origin. Let $f:\mathbb{D}_r\to\mathbb{C}$ be the function defined by
\begin{equation}
f(z)=z+\sum_{n\geq 2}a_{n}z^{n},  \label{eq0}
\end{equation}
where $r$ is less or equal than the radius of convergence of the above power series. Denote by $\mathcal{A}$ the class of allanalytic functions of the form \eqref{eq0}, that is, normalized by the conditions $f(0)=f^{\prime}(0)-1=0.$ We say that the function $f,$ defined by \eqref{eq0}, is starlike function in $\mathbb{D}_r$ if $f$ is univalent in $\mathbb{D}_r$, and the image domain $f(\mathbb{D}_r)$ is a starlike domain in $\mathbb{C}$ with respect to the origin (see \cite{Dur} for more details). Analytically, the function $f$ is starlike in $\mathbb{D}_r$ if and only if $$\real\left( \frac{zf^{\prime }(z)}{f(z)}\right) >0 \quad \mbox{for all}\ \ z\in\mathbb{D}_r.$$ For $\alpha \in [0,1)$ we say that the function $f$ is starlike of order $\alpha $ in $\mathbb{D}_r$ if and only if 
$$\real\left( \frac{zf^{\prime }(z)}{f(z)}\right) >\alpha \quad \mbox{for all}\ \ z\in\mathbb{D}_r.$$
The radius of starlikeness of order $\alpha$ of the function $f$ is defined as the real number
\begin{equation*}
r_{\alpha }^{\star}(f)=\sup \left\{ r>0 \left|\real
\left( \frac{zf^{\prime }(z)}{f(z)}\right)  >\alpha \;\text{for all }z\in
\mathbb{D}_r\right.\right\}.
\end{equation*}
Note that $r^{\star}(f)=r_{0}^{\star}(f)$ is in fact the largest radius such that the image region $f(\mathbb{D}_{r^{\star}(f)})$ is a starlike domain with respect to the origin.

The function $f,$ defined by \eqref{eq0}, is convex in the disk $\mathbb{D}_r$ if $f$ is univalent in $\mathbb{D}_r$, and the image domain $f(\mathbb{D}_r)$ is a convex domain in $\mathbb{C}.$ Analytically, the function $f$ is convex in $\mathbb{D}_r$ if and only if
$$\real\left(  1+\frac{zf^{\prime \prime }(z)}{f^{\prime }(z)}\right)>0  \quad \mbox{for all}\ \ z\in\mathbb{D}_r.$$
For $\alpha \in[0,1)$ we say that the function $f$ is convex of order $\alpha $ in $\mathbb{D}_r$ if and only if 
$$\real\left( 1+\frac{zf^{\prime \prime }(z)}{f^{\prime }(z)}\right)
>\alpha \quad \mbox{for all}\ \ z\in\mathbb{D}_r.$$ 
We shall denote the radius of convexity of order $\alpha $ of the function $f$ by the real number
\begin{equation*}
r_{\alpha }^{c}(f)=\sup \left\{ r>0 \left|\real\left( 1+
\frac{zf^{\prime \prime }(z)}{f^{\prime }(z)}\right) >\alpha \;\text{for all }
z\in\mathbb{D}_r\right.\right\} .
\end{equation*}
Note that $r^{c}(f)=r_{0}^{c}(f)$ is the largest radius such that the image region $f(\mathbb{D}_{r^{c}(f)})$ is a convex domain.

We recall that a real entire function $q$ belongs to the  Laguerre-P\'{o}lya class $\mathcal{LP}$ if it can be represented in the form $$q(x)=cx^{m}e^{-ax^2+bx}\prod_{n\geq1}\left(1+\frac{x}{x_n}\right)e^{-\frac{x}{x_n}},$$ with $c,b,x_n\in\mathbb{R}, a\geq0, m\in\mathbb{N}_0$ and $\sum1/{x_n}^2<\infty.$ We note that the class $\mathcal{LP}$ is the complement of the space of polynomials whose zeros are all real in the topology induced by the uniform convergence on the compact sets of the complex plane of polynomials with only real zeros. For more details on the class $\mathcal{LP}$ we refer to \cite[p. 703]{DC} and to the references therein.

Finally, let us take a look at some lemmas which are very useful in building our main results.

\begin{lemma}[ see \cite{DSz}]\label{lem}
If  $a>b>r\geq \left|z\right|,$ and $ \lambda\in [0,1] $, then
\begin{equation}\label{1.2}
\left|\frac{z}{b-z}-\lambda\frac{z}{a-z}\right|\leq\frac{r}{b-r}-\lambda\frac{r}{a-r}.
\end{equation}
The followings can be obtained as a natural consequence of this inequality:
\begin{equation}\label{lemineq1}
\real\left(\frac{z}{b-z}-\lambda\frac{z}{a-z}\right)\leq\frac{r}{b-r}-\lambda\frac{r}{a-r}
\end{equation}
and 
\begin{equation}\label{lemineq2}
\real\left(\frac{z}{b-z}\right)\leq\left|\frac{z}{b-z}\right|\leq\frac{r}{b-r}.
\end{equation}
\end{lemma}

The result of Runckel \cite[Thm. 4]{Runckel}, stated in \Cref{Runckel}, play a key role in proving the reality of zeros of some special functions. For interesting applications of this lemma one can refer to the study of Baricz and Sanjeev \cite{BSan}.
\begin{lemma}[Runckel, 1969]\label{Runckel}
If $f(z)=\sum_{n\geq 0}a_nz^n$ can be represented by $f(z)=e^{az^2}h(z),$ where $a\leq0$ and $h$ is of form
\[h(z)=ce^{bz}\prod_{n\geq 1}\left( 1-\frac{z}{c_n}\right)e^{\frac{z}{c_n}}, \text{ \ \ }a,b\in\mathbb{R}, \sum_{n \geq 1}\left| c_n \right|^{-2}<\infty,   \]
$f$ has real zeros only (or no zeros at all), and $G$ is of form
\[ G(z)=e^{\beta z}\prod_{n\geq1}\left(1+\frac{z}{\alpha_n}\right)e^{-\frac{z}{\alpha_n}}, \text{ \ \ }\alpha_n>0, \beta\in\mathbb{R}, \sum_{n \geq 1}\alpha_n^{-2}<\infty,  \]
then the function $\sum_{n\geq0}a_nG(n)z^{n}$ has real zeros only.
\end{lemma}

\section{\bf The radii of starlikeness and convexity of generalized Struve functions}
\setcounter{equation}{0}
In this section we focus on the generalized Struve functions, which is similar to the generalized Bessel functions of Wright type (and of Galue type), defined as
$$_{q}W_{p,b,c,\delta}(z)=\sum_{n\geq 0}\frac{(-1)^{n}c^n}{n!\Gamma(qn+\frac{p}{\delta}+\frac{b+2}{2})}\left( \frac{z}{2}\right) ^{2n+p+1}, \quad (q\in \mathbb{N};\quad p,b,c\in \mathbb{C})$$
where $\delta>0.$

It is clear that we can derive a number of well-known special functions from  generalized Struve function for some values of the parameters. Some of them is as follows:
$${}_{1}W_{\nu-1,2,1,1}(z)=J_{\nu}(z),$$
where $J_{\nu}(z)$ is Bessel function of first kind \cite{Wat}.
$${}_{q}W_{2\nu+2\lambda-1,1,1,2}(z)=\left( \frac{z}{2}\right)^{\nu} \frac{\Gamma(\lambda+n+1)}{\Gamma(1+n)}J_{\nu,\lambda}^{q}(z),$$
where $J_{\nu,\lambda}^{q}(z)$ is Bessel-maitland function \cite{Pathak}.
$${}_{q}W_{p-1,1,-1,1}(z)={}_{q}I_{p}(z),$$
where ${}_{q}I_{p}(z)$ is the Galue type generalization of modified Bessel function \cite{Baricz}.

It is obvious that the function $z\mapsto{}_{q}W_{p,b,c,\delta}(z)$ does not belong to the class $\mathcal{A},$ and thus first we perform some natural normalization. We define three functions originating from $_{q}W_{p,b,c,\delta}(.)$

\begin{align*}
{}_{q}f_{p,b,c,\delta}(z)=& \bigg(2^{p+1}\Gamma(\frac{p}{\delta}+\frac{b+2}{2}){}_{q}W_{p,b,c,\delta}(z)\bigg)^{\frac{1}{p+1}}, \\
{}_{q}g_{p,b,c,\delta}(z)=&2^{p+1} \Gamma(\frac{p}{\delta}+\frac{b+2}{2})z^{-p}{}_{q}W_{p,b,c,\delta}(z),\\
{}_{q}h_{p,b,c,\delta}(z)=& 2^{p+1}\Gamma(\frac{p}{\delta}+\frac{b+2}{2})z^{1-\frac{p+1}{2}}{}_{q}W_{p,b,c,\delta}(\sqrt{z}).
\end{align*}
Observe that these functions belong to the class  $\mathcal{A}.$ No doubt it is possible to write infinitely many other normalizations; the main motivation to consider the above ones is the fact that their particular cases in terms of Bessel functions appear in the literature or are similar to the studied normalization in the literature. In this study, we are going to present some intriguing results on these three functions by using some basic techniques of geometric function theory.

The following lemma, which we believe is of independent interest, has a crucial role in the proof of our main results.
\begin{lemma} \label{mainlemma}
Let $\delta,q,b,c>0$ and $p+1>0$. Then $z \mapsto {}_{q}W_{p,b,c,\delta}(z) $ possesses infinitely many zeros which are all real. Denoting by ${}_{q}\omega_{p,b,c,n}$ the $n$th positive zero of $_{q}W_{p,b,c,\delta}(z)$, under the same conditions the Weierstrassian decomposition
\begin{equation}\label{mainlemma1}
2^{p+1}\Gamma(\frac{p}{\delta}+\frac{b+2}{2}){}_{q}W_{p,b,c,\delta}(z)=z^{p+1}\prod_{n\geq 1}\left(1-\frac{z^2}{{}_{q}\omega_{p,b,c,\delta,n}^2}\right) 
\end{equation}
is fulfilled, and this product is uniformly convergent on compact subsets of the complex
plane. Moreover, if we denote by ${}_{q}\omega_{p,b,c,\delta,n}^{\prime}$ the nth positive zero of ${}_{q}W_{p,b,c,\delta}^{\prime}(z)$, then positive zeros of ${}_{q}W_{p,b,c,\delta}$ are interlaced with those of ${}_{q}W_{p,b,c,\delta}^{\prime}.$ In other words, the zeros satisfy the chain of inequalities
$${}_{q}\omega_{p,b,c,\delta,1}^{\prime}<{}_{q}\omega_{p,b,c,\delta,1}<{}_{q}\omega_{p,b,c,\delta,2}^{\prime}<{}_{q}\omega_{p,b,c,\delta,2}<\dots_{.}$$
\end{lemma}
\begin{proof}
Let us prove the reality of zeros of the function ${} _{q}W_{p,b,c,\delta}(z).$ Consider the entire function 
$${}_{q}W_{p,b,c,\delta}(z)=\left(\frac{z}{2} \right)^{p+1} \sum_{n\geq0}\frac{(-1)^{n} c^n}{n!\Gamma(qn+\frac{p}{\delta}+\frac{b+2}{2})}\left(\frac{z}{2} \right)^{2n} .$$
The function ${}_{q}G_{p,b,c,\delta}:\left[0, \infty \right) \rightarrow \mathbb{R}$ defined by
$${}_{q}G_{p,b,c,\delta}(z)=\frac{1}{\Gamma(qz+\frac{p}{\delta}+\frac{b+2}{2})}$$
is entire function and of growth order $1,$ belongs to $\mathcal{LP}.$ If we choose $f(z)=e^{-c\left(\tfrac{z}{2}\right) ^2},$ then by making use of the Runckel's above-mentioned result (that is, \Cref{Runckel}) we obtain that the function ${}_{q}W_{p,b,c,\delta}(z)$ has real zeros only if $\delta,q,b,c>0$ and $p+1>0.$ Moreover, the growth order of the generalized Struve function is calculated in the following:
\begin{align*}
\rho&=\limsup_{n\rightarrow \infty}\frac{n \log n}{-\log\left|c_n \right| }=\limsup_{n\rightarrow \infty}\frac{n \log n}{-\log\left|\frac{(-c)^{n}}{n!\Gamma(qn+\frac{p}{\delta}+\frac{b+2}{2})}\right|} \\
&=\limsup_{n\rightarrow \infty}\frac{1}{-\frac{n \log \left| c\right| }{n log n}+\frac{\log \Gamma(n+1) }{n log n}+\frac{\log \Gamma(qn+\frac{p}{\delta}+\frac{b+2}{2})}{n log n}}\\
&=\frac{1}{1+q}.
\end{align*}
It is evident that $0<\rho=\frac{1}{1+q}<1,$ for $q>0.$ It is well known that if the finite growth order $\rho$ of an entire function is not an integer, then the function has infinitely many zeros. In that case, by virtue of the Hadamard theorem on growth order of the entire function, it follows that its infinite product representation is exactly what we have in \Cref{mainlemma}. Taking into account the infinite product representation we get
\begin{equation}\label{Realzeros1}
\frac{{}_{q}W_{p,b,c,\delta}'(z)}{_{q}W_{p,b,c,\delta}(z)}=\frac{p+1}{z}-\sum_{n \geq 1}\frac{2z}{{}_{q}\omega_{p,b,c,\delta,n}^2 -z^2}.
\end{equation}
Diferentiating both sides of \eqref{Realzeros1} we arrive at
$$\frac{d}{dz}\left(\frac{_{q}W_{p,b,c,\delta}'(z)}{_{q}W_{p,b,c,\delta}(z)} \right)=-\frac{p+1}{z^2}-2\sum_{n \geq 1}\frac{{}_{q}\omega_{p,b,c,\delta,n}^2+z^2}{\left( {}_{q}\omega_{p,b,c,\delta,n}^2-z^2\right) ^2}.$$
Since the expression on the right-hand side is real and negative for $z$ real, the quotient $\frac{_{q}W_{p,b,c,\delta}'(z)}{_{q}W_{p,b,c,\delta}(z)}$ is a strictly decreasing function from $+\infty$ to $-\infty$ as $z$ increases through real values over the open interval $({}_{q}\omega_{p,b,c,\delta,n}, {}_{q}\omega_{p,b,c,\delta,n+1}) \text{ \ \ } n\in \mathbb{N}.$ That is to say that the function $_{q}W_{p,b,c,\delta}'(z)$ vanishes just once between two consecutive zeros of the function $_{q}W_{p,b,c,\delta}(z).$

The lemma is proved.
\end{proof}
\subsection{The radii of starlikeness of order $\alpha$ of the functions ${}_{q}f_{p,b,c,\delta}$, ${}_{q}g_{p,b,c,\delta}$ and ${}_{q}h_{p,b,c,\delta}$.}
Our first main result is related to the radii of starlikeness of these three normalized generalized Struve functions. In other words, the aim of this section is to determine the radii of starlikeness of order \(\alpha \) of the normalized Struve functions and to give tight lower and upper bounds for radii of starlikeness of the normalized Struve functions.

\begin{theorem}\label{MainTheo1}
Let $\delta,q,b,c>0$, $p+1>0$ and $\alpha\in\left[ 0,1\right).$ Then the following assertions are true.
\item [\bf a.] The radius of starlikeness of order $\alpha$ of the function $_{q}f_{p,b,c,\delta}$ is $r^{\star}_{\alpha}(_{q}f_{p,b,c,\delta})={}_{q}x_{p,b,c,\delta,1},$ where ${}_{q}x_{p,b,c,\delta,1}$ is the smallest zero of the equation
\[  r {}_{q}W_{p,b,c,\delta}'(r)-\alpha(p+1){}_{q}W_{p,b,c,\delta}(r)=0  .\]
\item[\bf b.] The radius of starlikeness of order $\alpha$ of the function $_{q}g_{p,b,c,\delta}$ is $r^{\star}_{\alpha}(_{q}g_{p,b,c,\delta})={}_{q}y_{p,b,c,\delta,1},$ where ${}_{q}y_{p,b,c,\delta,1}$ is the smallest zero of the equation
\[ r {}_{q}W_{p,b,c,\delta}'(r)-(\alpha+p){}_{q}W_{p,b,c,\delta}(r)=0 .\]
\item[\bf c.] The radius of starlikeness of order $\alpha$ of the function $_{q}h_{p,b,c,\delta}$ is $r^{\star}_{\alpha}(_{q}h_{p,b,c,\delta})={}_{q}z_{p,b,c,\delta,1},$ where ${}_{q}z_{p,b,c,\delta,1}$ is the smallest zero of the equation
\[\sqrt{r}{}_{q}W_{p,b,c,\delta}'(\sqrt{r})-2\big(\alpha+\frac{p+1}{2}-1\big){}_{q}W_{p,b,c,\delta}(\sqrt{r})=0. \]
\end{theorem}
\begin{proof}
We shall show that the inequalities
\begin{equation}\label{MT1}
\real\left( \frac{z{}_{q}f_{p,b,c,\delta}'(z)}{_{q}f_{p,b,c,\delta}(z)}\right)>\alpha, \quad \real\left(\frac{z{}_{q}g_{p,b,c,\delta}'(z)}{_{q}g_{p,b,c,\delta}(z)}\right)>\alpha \text{ \ and \ } \real\left( \frac{z{}_{q}h_{p,b,c,\delta}'(z)}{_{q}h_{p,b,c,\delta}(z)}\right)>\alpha,
\end{equation}
are valid for $z\in \mathbb{D}_{{}_{q}x_{p,b,c,\delta,1}}(_{q}f_{p,b,c,\delta}),$  $z\in \mathbb{D}_{{}_{q}y_{p,b,c,\delta,1}}(_{q}g_{p,b,c,\delta})$ and  $z\in \mathbb{D}_{{}_{q}z_{p,b,c,\delta,1}}(_{q}h_{p,b,c,\delta}),$ respectively, and
each of the above-mentioned inequalities does not hold in any larger disk. Consider the functions
\begin{align*}
{}_{q}f_{p,b,c,\delta}(z)=& \bigg(2^{p+1}\Gamma(\frac{p}{\delta}+\frac{b+2}{2}){}_{q}W_{p,b,c,\delta}(z)\bigg)^{\frac{1}{p+1}}, \\
{}_{q}g_{p,b,c,\delta}(z)=&2^{p+1} \Gamma(\frac{p}{\delta}+\frac{b+2}{2})z^{-p}{}_{q}W_{p,b,c,\delta}(z),\\
{}_{q}h_{p,b,c,\delta}(z)=& 2^{p+1}\Gamma(\frac{p}{\delta}+\frac{b+2}{2})z^{1-\frac{p+1}{2}}{}_{q}W_{p,b,c,\delta}(\sqrt{z}).
\end{align*}
As a result of the logarithmic derivation we arrive at
\begin{align*}
\frac{z{}_{q}f_{p,b,c,\delta}'(z)}{{}_{q}f_{p,b,c,\delta}(z)}=& \frac{1}{p+1}\left(\frac{z{}_{q}W_{p,b,c,\delta}'(z)}{{}_{q}W_{p,b,c,\delta}(z)} \right)=1-\frac{1}{p+1}\sum_{n \geq 1}\frac{2z^2}{{}_{q}\omega_{p,b,c,\delta,n}^2-z^2},  \\
\frac{z{}_{q}g_{p,b,c,\delta}'(z)}{{}_{q}g_{p,b,c,\delta}(z)}=&-p+\left(\frac{z{}_{q}W_{p,b,c,\delta}'(z)}{{}_{q}W_{p,b,c,\delta}(z)} \right)= 1-\sum_{n \geq 1}\frac{2z^2}{{}_{q}\omega_{p,b,c,\delta,n}^2-z^2}, \\
\frac{z{}_{q}h_{p,b,c,\delta}'(z)}{{}_{q}h_{p,b,c,\delta}(z)}=&1-\frac{p+1}{2}+\frac{1}{2}\left(\frac{\sqrt{z}{}_{q}W_{p,b,c,\delta}'(\sqrt{z})}{{}_{q}W_{p,b,c,\delta}(\sqrt{z})} \right)=1-\sum_{n \geq 1}\frac{z}{{}_{q}\omega_{p,b,c,\delta,n}^2-z}.
\end{align*}
By making use of Eq. \eqref{lemineq2} given in \Cref{lem} we have
\begin{align*}
\real\left(\frac{z{}_{q}f_{p,b,c,\delta}'(z)}{_{q}f_{p,b,c,\delta}(z)}\right)&=1-\frac{1}{p+1}\real\left(\sum_{n \geq 1}\frac{2z^2}{{}_{q}\omega_{p,b,c,\delta,n}^2-z^2} \right)  
\\
&\geq 1-\frac{1}{p+1}\sum_{n \geq 1}\frac{2\left| z\right| ^2}{{}_{q}\omega_{p,b,c,\delta,n}^2-\left| z\right| ^2} = \frac{\left| z\right| {}_{q}f_{p,b,c,\delta}'(\left| z\right| )}{_{q}f_{p,b,c,\delta}(\left| z\right| )},\\
\real\left(\frac{z{}_{q}g_{p,b,c,\delta}'(z)}{{}_{q}g_{p,b,c,\delta}(z)}\right)&=1-\real\left(\sum_{n \geq 1}\frac{2z^2}{{}_{q}\omega_{p,b,c,\delta,n}^2-z^2} \right) 
\\
&\geq 1- \sum_{n \geq 1}\frac{2\left| z\right| ^2}{{}_{q}\omega_{p,b,c,\delta,n}^2-\left| z\right| ^2}= \frac{\left| z\right| {}_{q}g_{p,b,c,\delta}'(\left| z\right| )}{{}_{q}g_{p,b,c,\delta}(\left| z\right|)},
\\
\real\left( \frac{z{}_{q}h_{p,b,c,\delta}'(z)}{_{q}h_{p,b,c,\delta}(z)}\right)&=1-\real\left( \sum_{n \geq 1}\frac{z}{{}_{q}\omega_{p,b,c,\delta,n}^2-z}\right)\\
&\geq 1- \sum_{n \geq 1}\frac{\left| z\right| }{{}_{q}\omega_{p,b,c,\delta,n}^2-\left| z\right| }=\frac{\left| z\right| {}_{q}h_{p,b,c,\delta}'(\left| z\right| )}{{}_{q}h_{p,b,c,\delta}(\left| z\right| )}.
\end{align*}
It is important to mention that equalities in the above-mentioned inequalities are attained only when $z=\left| z\right|=r.$ In light of the later inequalities and the minimum principle for harmonic functions we deduce that the inequalities stated in \eqref{MT1} hold if and only if $\left|z\right|<{}_{q}x_{p,b,c,\delta,1}, $ $\left|z\right|<{}_{q}y_{p,b,c,\delta,1} $ and $\left|z\right|<{}_{q}z_{p,b,c,\delta,1} $, respectively, where ${}_{q}x_{p,b,c,\delta,1},$ ${}_{q}y_{p,b,c,\delta,1}$ and ${}_{q}z_{p,b,c,\delta,1}$ are the smallest positive roots of the equations
\begin{equation} \label{MT2}
 \frac{r{}_{q}f_{p,b,c,\delta}'(r)}{_{q}f_{p,b,c,\delta}(r)}=\alpha, \text{ \ \ } \frac{r{}_{q}g_{p,b,c,\delta}'(r)}{_{q}g_{p,b,c,\delta}(r)}=\alpha \text{ \ and \ } \frac{r{}_{q}h_{p,b,c,\delta}'(r)}{_{q}h_{p,b,c,\delta}(r)}=\alpha.
\end{equation}

As a natural result of these equalities we conclude that 
\[ r {}_{q}W_{p,b,c,\delta}'(r)-\alpha(p+1){}_{q}W_{p,b,c,\delta}(r)=0, \text{ \ \ } r {}_{q}W_{p,b,c,\delta}'(r)-(\alpha+p){}_{q}W_{p,b,c,\delta}(r)=0 \]
and
\[\sqrt{r}{}_{q}W_{p,b,c,\delta}'(\sqrt{r})-2\big(\alpha+\frac{p+1}{2}-1\big){}_{q}W_{p,b,c,\delta}(\sqrt{r})=0. \]
In order to finish the proof we need to show that the roots ${}_{q}x_{p,b,c,\delta,1}, $ ${}_{q}y_{p,b,c,\delta,1} $ and ${}_{q}z_{p,b,c,\delta,1} $ are, respectively, the smallest zeros of the above-mentioned transcendental equation. To do this, in view of the above inequalities, consider the functions for $r\in \left( 0,{}_{q}w_{p,b,c,\delta,1}\right) $
\begin{align*}
\inf_{z\in \mathbb{D}_{r}} \real\left(\frac{z{}_{q}f_{p,b,c,\delta}'(z)}{_{q}f_{p,b,c,\delta}(z)}\right)=&\frac{r{}_{q}f_{p,b,c,\delta}'(r)}{{}_{q}f_{p,b,c,\delta}(r)}={}_{q}u_{p,b,c,\delta}(r),\\
\inf_{z\in \mathbb{D}_{r}}\real\left( \frac{z{}_{q}g_{p,b,c,\delta}'(z)}{{}_{q}g_{p,b,c,\delta}(z)}\right) =&\frac{r{}_{q}g_{p,b,c,\delta}'(r)}{_{q}g_{p,b,c,\delta}(r)}={}_{q}v_{p,b,c,\delta}(r), \\
\inf_{z\in\mathbb{D}_{r}}\real\left( \frac{z{}_{q}h_{p,b,c,\delta}'(z)}{_{q}h_{p,b,c,\delta}(z)}\right) =&\frac{r{}_{q}h_{p,b,c,\delta}'(r)}{_{q}h_{p,b,c,\delta}(r)}={}_{q}w_{p,b,c,\delta}(r).
\end{align*}
Taking into account facts that
\[ {}_{q}u_{p,b,c,\delta}'(r)=-\frac{1}{p+1}\sum_{n \geq 1}\frac{4z{}_{q}\omega_{p,b,c,\delta,n}^{2}}{\left({}_{q}w_{p,b,c,\delta,n}^{2}-z^2 \right)^2 }<0, \text{ \ \ }{}_{q}v_{p,b,c,\delta}'(r)=-\sum_{n \geq 1}\frac{4z{}_{q}\omega_{p,b,c,\delta,n}^{2}}{\left({}_{q}\omega_{p,b,c,\delta,n}^{2}-z^2 \right)^2 }<0 \]
and
\[ {}_{q}w_{p,b,c,\delta}(r)=-\sum_{n \geq 1}\frac{{}_{q}\omega_{p,b,c,\delta,n}^{2}}{\left({}_{q}\omega_{p,b,c,\delta,n}^{2}-z \right)^2 }<0\]
we deduce that the real functions ${}_{q}u_{p,b,c,\delta}, {}_{q}v_{p,b,c,\delta}, {}_{q}w_{p,b,c,\delta}:\left(0,{}_{q}\omega_{p,b,c,\delta,1} \right)\rightarrow \mathbb{R} $ are strictly decreasing. In addition, by virtue of the limits determined as
\[ \lim_{r \searrow 0}{}_{q}u_{p,b,c,\delta}(r)=\lim_{r \searrow 0}{}_{q}v_{p,b,c,\delta}(r)=\lim_{r \searrow 0}{}_{q}w_{p,b,c,\delta}(r)=1\]
and
\[ \lim_{r\nearrow {}_{q}\omega_{p,b,c,\delta,1}}{}_{q}u_{p,b,c,\delta}(r)=\lim_{r\nearrow {}_{q}\omega_{p,b,c,\delta,1}}{}_{q}v_{p,b,c,\delta}(r)=\lim_{r\nearrow {}_{q}\omega_{p,b,c,\delta,1}}{}_{q}w_{p,b,c,\delta}(r)=-\infty. \]
we say that the roots ${}_{q}x_{p,b,c,\delta,1}, $ ${}_{q}y_{p,b,c,\delta,1} $ and ${}_{q}z_{p,b,c,\delta,1} $ are, respectively, the smallest zeros of the transcendental equation given in \eqref{MT2}. It is clear that these are desired results.
\end{proof}

The following theorems include some tight lower and upper bounds for the radii
of starlikeness of the functions considered in the above theorems.
\begin{theorem}\label{MainTheo1'}
Let $\delta,q,b,c>0$ and $p+1>0.$ The radius of starlikeness $r^{\star}({}_{q}f_{p,b,c,\delta})$ is satisfies
\[\text{\footnotesize$\sqrt{\frac{2(p+1)\Gamma(q+\frac{p}{\delta}+\frac{b+2}{2})}{c(p+3)\Gamma(\frac{p}{\delta}+\frac{b+2}{2})}}<r^{\star}({}_{q}f_{p,b,c,\delta})<2\sqrt{\frac{(p+1)(p+3)\Gamma(q+\frac{p}{\delta}+\frac{b+2}{2})\Gamma(2q+\frac{p}{\delta}+\frac{b+2}{2})}{c\left\lbrace(p+3)^2\Gamma(\frac{p}{\delta}+\frac{b+2}{2})\Gamma(2q\frac{p}{\delta}+\frac{b+2}{2})-2(p+5)(p+1)\Gamma^2(q+\frac{p}{\delta}+\frac{b+2}{2})\right\rbrace }}$}.\]
\end{theorem}
\begin{proof}
When \(\alpha=0 \) in \autoref{MainTheo1}, we have $r^{\star}({}_{q}f_{p,b,c,\delta})$ is the smallest positive root of the equation \({}_{q}W_{p,b,c,\delta}'(z)=0.\) That is, the radius of starlikeness of the normalized generalized Struve function \({}_{q}f_{p,b,c,\delta}(z)\) corresponds to the radius of starlikeness of the function \({}_{q}\Xi_{p,b,c,\delta}(z)={}_{q}W_{p,b,c,\delta}'(z)\). The infinite series representations of the function \({}_{q}\Xi_{p,b,c,\delta}(z)\) and its derivative are given as
\[{}_{q}\Xi_{p,b,c,\delta}(z)=\sum_{n\geq 0}\frac{(-1)^{n}c^n(2n+p+1)}{2^{2n+p+1}n!\Gamma(qn+\frac{p}{\delta}+\frac{b+2}{2})}z^{2n+p}\]
and
\[{}_{q}\Xi_{p,b,c,\delta}'(z)=\sum_{n\geq 0}\frac{(-1)^{n}c^n(2n+p)(2n+p+1)}{2^{2n+p+1}n!\Gamma(qn+\frac{p}{\delta}+\frac{b+2}{2})}z^{2n+p-1}.\]
Because of the facts that the function  \({}_{q}W_{p,b,c,\delta}\) is of the Laguerre-P\'olya class of entire functions and that the class \(\mathcal{LP}\) is closed under differentiation, we deduce that the function \({}_{q}\Xi_{p,b,c,\delta}\) is also in the class \(\mathcal{LP}\). This means that the zeros of the functions \({}_{q}\Xi_{p,b,c,\delta}\) are all real. If we denote of the zeros of the function \({}_{q}\Xi_{p,b,c,\delta}\) by \({}_{q}\varepsilon_{p,b,c,\delta,n}\), then the infinite product representation of the function \({}_{q}\Xi_{p,b,c,\delta}\) is given as follows:
\[2^{p+1}\Gamma(\frac{p}{\delta}+\frac{b+2}{2}){}_{q}\Xi_{p,b,c,\delta}(z)=(p+1)z^{p}\prod_{n\geq1}\left(1-\frac{z^2}{{}_{q}\varepsilon_{p,b,c,\delta,n}^2} \right).\]
With the help of the logarithmic derivation of the last equality, we obtain
\[\frac{{}_{q}\Xi_{p,b,c,\delta}'(z)}{{}_{q}\Xi_{p,b,c,\delta}(z)}-\frac{p}{z}=-2\sum_{n\geq 1}\frac{z}{{}_{q}\varepsilon_{p,b,c,\delta,n}^2-z^2}=-2\sum_{n\geq 1}\sum_{k\geq0}\frac{z^{2k+1}}{{}_{q}\varepsilon_{p,b,c,\delta,n}^{2k+2}}=-2\sum_{k\geq0}\tau_{k+1}z^{2k+1}, \quad\left| z\right|<{}_{q}\varepsilon_{p,b,c,\delta,1},\]
where \(\tau_{k}=\sum_{n\geq 1}{}_{q}\varepsilon_{p,b,c,\delta,n}^{-2k}.\) On the other hand, by making use of the infinite sum representation of the function \({}_{q}\Xi_{p,b,c,\delta}\) we arrive at
\[\frac{{}_{q}\Xi_{p,b,c,\delta}'(z)}{{}_{q}\Xi_{p,b,c,\delta}(z)}=\sum_{n\geq 0}\frac{(-1)^{n}c^n(2n+p)(2n+p+1)}{2^{2n+p+1}n!\Gamma(qn+\frac{p}{\delta}+\frac{b+2}{2})}z^{2n+p-1} \bigg/  \sum_{n\geq 0}\frac{(-1)^{n}c^n(2n+p+1)}{2^{2n+p+1}n!\Gamma(qn+\frac{p}{\delta}+\frac{b+2}{2})}z^{2n+p}.\]
By comparing the coefficients of the last two equalities, we have
\[ \tau_{1}=\frac{c(p+3)\Gamma(\frac{p}{\delta}+\frac{b+2}{2})}{4(p+1)\Gamma(q+\frac{p}{\delta}+\frac{b+2}{2})} \text{ \ and \ } \tau_{2}=\frac{c^2(p+3)^2\Gamma^2(\frac{p}{\delta}+\frac{b+2}{2})}{16(p+1)^2\Gamma^2(q+\frac{p}{\delta}+\frac{b+2}{2})}-\frac{c^2(p+5)\Gamma(\frac{p}{\delta}+\frac{b+2}{2})}{16(p+1)\Gamma(2q+\frac{p}{\delta}+\frac{b+2}{2})}. \]
Now, by using the Euler-Rayleigh inequalities \(\tau_{k}^{-\frac{1}{k}}<{}_{q}\varepsilon_{p,b,c,\delta,1}^2<\frac{\tau_{k}}{\tau_{k+1}}\) for and \(k\in \mathbb{N}\), we get the following inequality:
\[\text{\footnotesize$2\sqrt{\frac{(p+1)\Gamma(q+\frac{p}{\delta}+\frac{b+2}{2})}{c(p+3)\Gamma(\frac{p}{\delta}+\frac{b+2}{2})}}<r^{\star}({}_{q}f_{p,b,c,\delta})<2\sqrt{\frac{(p+1)(p+3)\Gamma(q+\frac{p}{\delta}+\frac{b+2}{2})\Gamma(2q+\frac{p}{\delta}+\frac{b+2}{2})}{c\left\lbrace(p+3)^2\Gamma(\frac{p}{\delta}+\frac{b+2}{2})\Gamma(2q+\frac{p}{\delta}+\frac{b+2}{2})-(p+1)(p+5)\Gamma^2(q+\frac{p}{\delta}+\frac{b+2}{2})\right\rbrace }}$}.\]
No doubt we can find more tighter bounds for other values of \(k\in\mathbb{N}. \)
\end{proof}
\begin{theorem}\label{MainTheo2}
Let $\delta,q,b,c>0$ and $p+1>0.$ The radius of starlikeness $r^{\star}({}_{q}g_{p,b,c,\delta})$ is satisfies
\[2\sqrt{\frac{\Gamma(q+\frac{p}{\delta}+\frac{b+2}{2})}{3c\Gamma(\frac{p}{\delta}+\frac{b+2}{2})}}<r^{\star}({}_{q}g_{p,b,c,\delta})<2\sqrt{\frac{3\Gamma(q+\frac{p}{\delta}+\frac{b+2}{2})\Gamma(2q+\frac{p}{\delta}+\frac{b+2}{2})}{c\big(9\Gamma(\frac{p}{\delta}+\frac{b+2}{2})\Gamma(2q+\frac{p}{\delta}+\frac{b+2}{2})-5\Gamma^{2}(q+\frac{p}{\delta}+\frac{b+2}{2})\big)}}.\]
\end{theorem}
\begin{proof}
Our aim is to give more tight bounds for the radius of starlikeness $r^{\star}({}_{q}g_{p,b,c,\delta})$ by making use of the Euler-Rayleigh inequalities. Let us recall that the radius of starlikeness of the function \( {}_{q}g_{p,b,c,\delta}(z)\) is the first positive zero of its derivative, according to \cite{BKS,BOS}. We can draw conclusion from \Cref{Runckel} that the zeros of
\[{}_{q}g_{p,b,c,\delta}(z)=\Gamma(\tfrac{p}{\delta}+\tfrac{b+2}{2})\sum_{n\geq0}\frac{(-1)^n c^n}{n!\Gamma(qn+\tfrac{p}{\delta}+\tfrac{b+2}{2})}\left(\frac{z}{2} \right)^{2n+1} \]
all are real when $\delta,q,b,c>0$ and $p+1>0$. Consequently, this function belongs to the Laguerre-P\'olya class $\mathcal{LP}$ of real entire functions (see \cite{DC} for more details), which are uniform limits of real polynomials whose all zeros are real. Now, since the Laguerre-P\'olya class  $\mathcal{LP}$ is closed under differentiation, it follows that ${}_{q}g_{p,b,c,\delta}'$ belongs also to the Laguerre-P\'olya class and hence all of its zeros are real. Now, we consider the entire function
\[{}_{q}\varphi_{p,b,c,\delta}(z)={}_{q}g_{p,b,c,\delta}'(2\sqrt{z})=\Gamma(\tfrac{p}{\delta}+\tfrac{b+2}{2})\sum_{n\geq0}\frac{(-1)^n(c)^n(2n+1)}{n!\Gamma(qn+\tfrac{p}{\delta}+\tfrac{b+2}{2})}z^n .\]
Therefore in light of Laguerre's Lemma in \cite[Lem. 1, p. 2208]{BSan} we obtain that the entire function
\[{}_{q}\gamma_{p,b,c,\delta}(z)=\Gamma(\tfrac{p}{\delta}+\tfrac{b+2}{2})\sum_{n\geq0}\frac{c^n(2n+1)z^n  }{n!\Gamma(qn+\tfrac{p}{\delta}+\tfrac{b+2}{2})}\]
has real zeros. Further, it is evident that the function has negative zeros since the coefficients of ${}_{q}\gamma_{p,b,c,\delta}$ are all positive. Hence ${}_{q}\gamma_{p,b,c,\delta}(-z)$ have real and positive zeros. That is to say, the function ${}_{q}\varphi_{p,b,c,\delta}(z)$ has real and only positive zeros. Suppose that \({}_{q}\alpha_{p,b,c,\delta,n}\)'s are  the zeros of the function ${}_{q}\varphi_{p,b,c,\delta}(z).$ Thus, since the function $z\mapsto{}_{q}\varphi_{p,b,c,\delta}(z)={}_{q}g_{p,b,c,\delta}'(2\sqrt{z}) $ has growth order $\tfrac{1}{1+q}$ it can be represented by the product
\[ {}_{q}\varphi_{p,b,c,\delta}(z)=\prod_{n\geq 1}\left(1-\frac{z}{{}_{q}\alpha_{p,b,c,\delta,n}} \right)  \]
where ${}_{q}\alpha_{p,b,c,\delta,n}>0$ for each $n \in \mathbb{N}.$ Now, by using the Euler-Rayleigh sum $\ell_{k}=\sum_{n \geq 1}{}_{q}\alpha_{p,b,c,\delta,n}^{-k} $ we obtain that
\begin{equation}\label{Euler-Rayleigh1}
\text{\footnotesize$\frac{{}_{q}\varphi_{p,b,c,\delta}'(z)}{{}_{q}\varphi_{p,b,c,\delta}(z)}=-\sum_{n \geq 1}\frac{1}{{}_{q}\alpha_{p,b,c,\delta,n}-z}=-\sum_{n \geq 1}\sum_{k\geq 0} \frac{z^k}{{}_{q}\alpha_{p,b,c,\delta,n}^{k+1}}=-\sum_{k\geq 0} \ell_{k+1}z^k, \text{ \ \ } \left| z\right| <{}_{q}\alpha_{p,b,c,\delta,1},$}
\end{equation}

\begin{equation}\label{Euler-Rayleigh2}
\frac{{}_{q}\varphi_{p,b,c,\delta}'(z)}{{}_{q}\varphi_{p,b,c,\delta}(z)}=\sum_{n\geq 0}\frac{(-1)^{n+1}c^{n+1}(2n+3)}{n!\Gamma(q(n+1)+\tfrac{p}{\delta}+\tfrac{b+2}{2})}z^n\bigg/ \sum_{n\geq 0}\frac{(-1)^n c^n (2n+1) }{n!\Gamma(qn+\tfrac{p}{\delta}+\tfrac{b+2}{2})}z^n .
\end{equation}
Equations \eqref{Euler-Rayleigh1} and \eqref{Euler-Rayleigh2} enable us to express the Euler-Rayleigh sums in terms of $p,b,c,\delta$ and by using the
Euler-Rayleigh inequalities $\ell_{k}^{-\frac{1}{k}}<{}_{q}\alpha_{p,b,c,\delta,1}<\frac{\ell_{k}}{\ell_{k+1}}$ we get the inequalities for $\delta,q,b,c>0$, $p+1>0$ and $k\in\mathbb{N}$
\[2\sqrt{\ell_{k}^{-\frac{1}{k}}}<r^{\star}({}_{q}g_{p,b,c,\delta})<2\sqrt{\frac{\ell_{k}}{\ell_{k+1}}}.\]
Since
\[\ell_{1}=\frac{3c\Gamma(\frac{p}{\delta}+\frac{b+2}{2})}{\Gamma(q+\frac{p}{\delta}+\frac{b+2}{2})} \text{ \ and \ } \ell_{2}=\frac{9c^2\Gamma^{2}(\frac{p}{\delta}+\frac{b+2}{2})}{\Gamma^{2}(q+\frac{p}{\delta}+\frac{b+2}{2})}-\frac{5c^2\Gamma(\frac{p}{\delta}+\frac{b+2}{2})}{\Gamma(2q+\frac{p}{\delta}+\frac{b+2}{2})} \]
in particular, for $k=1$ from the above Euler-Rayleigh inequalities we have the next inequality for $2\sqrt{{}_{q}\alpha_{p,b,c,\delta,1}},$ that is,
\[2\sqrt{\frac{\Gamma(q+\frac{p}{\delta}+\frac{b+2}{2})}{3c\Gamma(\frac{p}{\delta}+\frac{b+2}{2})}}<r^{\star}({}_{q}g_{p,b,c,\delta})<2\sqrt{\frac{3\Gamma(q+\frac{p}{\delta}+\frac{b+2}{2})\Gamma(2q+\frac{p}{\delta}+\frac{b+2}{2})}{c\big(9\Gamma(\frac{p}{\delta}+\frac{b+2}{2})\Gamma(2q+\frac{p}{\delta}+\frac{b+2}{2})-5\Gamma^{2}(q+\frac{p}{\delta}+\frac{b+2}{2})\big)}}.\]
Of course, we can obtain more tighter bounds for other values of $k\in \mathbb{N}.$
\end{proof}
\begin{theorem}\label{MainTheo3}
The radius of starlikeness $r^{\star}({}_{q}h_{p,b,c,\delta})$ is satisfies
\[\frac{2\Gamma(q+\frac{p}{\delta}+\frac{b+2}{2})}{c\Gamma(\frac{p}{\delta}+\frac{b+2}{2})}<r^{\star}({}_{q}h_{p,b,c,\delta})<\frac{8\Gamma(q+\frac{p}{\delta}+\frac{b+2}{2})\Gamma(2q+\frac{p}{\delta}+\frac{b+2}{2})}{c\left\lbrace4\Gamma(\frac{p}{\delta}+\frac{b+2}{2})\Gamma(2q+\frac{p}{\delta}+\frac{b+2}{2})-3\Gamma^{2}(q+\frac{p}{\delta}+\frac{b+2}{2}) \right\rbrace } .\]

\end{theorem}
\begin{proof}
Consider the infinite sum representation of ${}_{q}h_{p,b,c,\delta}(z)$ and its derivative
\[ {}_{q}h_{p,b,c,\delta}(z)= \Gamma(\frac{p}{\delta}+\frac{b+2}{2})\sum_{n\geq 0}\frac{(-1)^n c^n}{4^n n!\Gamma(qn+\frac{p}{\delta}+\frac{b+2}{2})}z^{n+1},\]
\[ {}_{q}\Upsilon_{p,b,c,\delta}(z)={}_{q}h_{p,b,c,\delta}'(4z)=\Gamma(\frac{p}{\delta}+\frac{b+2}{2})\sum_{n\geq 0}\frac{(-1)^n c^n(n+1)}{ n!\Gamma(qn+\frac{p}{\delta}+\frac{b+2}{2})}z^{n} .\]
Moreover, it is possible to prove the reality of the zeros of the function $ {}_{q}h_{p,b,c,\delta}$ by using a similar approaching to the proof of \Cref{mainlemma}. This means that  $ {}_{q}h_{p,b,c,\delta}$ belongs to the Laguerre-P\'olya class $\mathcal{LP}.$ Therefore the function ${}_{q}h_{p,b,c,\delta}'$ belongs to the Laguerre-P\'olya class $\mathcal{LP}$ and has only real zeros. No doubt this is also valid for the function ${}_{q}\Upsilon_{p,b,c,\delta}.$ That is to say, the function ${}_{q}\Upsilon_{p,b,c,\delta}$ is the element of the Laguerre-P\'olya class $\mathcal{LP}.$ Therefore, we can reach the conclusion from Laguerre's Lemma stated in \cite[Lem. 1, p. 2208]{BSan} that the function ${}_{q}\Upsilon_{p,b,c,\delta}$ has only positive zeros and has growth order $\tfrac{1}{1+q}$, and thus ${}_{q}\Upsilon_{p,b,c,\delta}(z)$ can be represented by the product
\begin{equation}
{}_{q}\Upsilon_{p,b,c,\delta}(z)=\prod_{n\geq 1}\left(1-\frac{z}{{}_{q}\varsigma_{p,b,c,\delta,n}} \right),
\end{equation}
where \({}_{q}\varsigma_{p,b,c,\delta,n}\)'s are the zeros of the function \({}_{q}\Upsilon_{p,b,c,\delta}(z)\) and ${}_{q}\varsigma_{p,b,c,\delta,n}>0$ for each $n\in \mathbb{N}.$ Now, by using the Euler-Rayleigh sum $\kappa_{k}=\sum_{n \geq 1}{}_{q}\varsigma_{p,b,c,\delta,n}^{-k}$ the infinite sum representation of the generalized Struve function ${}_{q}W_{p,b,c,\delta}$ we get
\[\frac{{}_{q}\Upsilon_{p,b,c,\delta}'(z)}{{}_{q}\Upsilon_{p,b,c,\delta}(z)}=-\sum_{n \geq 1}\frac{1}{{}_{q}\varsigma_{p,b,c,\delta,n}-z} =-\sum_{n \geq 1}\sum_{k\geq 0}\frac{z^k}{{}_{q}\varsigma_{p,b,c,\delta,n}^{k+1}}=-\sum_{k\geq 0}\kappa_{k+1}z^k, \text{ \ \ }\left| z\right|<{}_{q}\varsigma_{p,b,c,\delta,1}, \]
\[ \frac{{}_{q}\Upsilon_{p,b,c,\delta}'(z)}{{}_{q}\Upsilon_{p,b,c,\delta}(z)}= \sum_{n\geq 0}\frac{(-1)^{n+1}c^{n+1}(n+2)}{n!\Gamma(q(n+1)+\frac{p}{\delta}+\frac{b+2}{2})}z^n \bigg/ \sum_{n\geq 0}\frac{(-1)^n c^{n}(n+1)}{n!\Gamma(qn+\frac{p}{\delta}+\frac{b+2}{2})}z^n.\]
With the help of the above-mentioned last two equality we can express the Euler-Rayleigh sums in terms of $p,b,c,\delta$ and by using the
Euler-Rayleigh inequalities $\kappa_{k}^{-\frac{1}{k}}<{}_{q}\varsigma_{p,b,c,\delta,1}<\frac{\kappa_{k}}{\kappa_{k+1}}$ we get the inequalities for $4{}_{q}\varsigma_{p,b,c,\delta,1}$ for $\delta,q,b,c>0$, $p+1>0$ and $k\in \mathbb{N}$
\[4\kappa_{k}^{-\frac{1}{k}}<r^{\star}({}_{q}h_{p,b,c,\delta})<4\frac{\kappa_{k}}{\kappa_{k+1}}.\]
Since
\[\kappa_{1}=\frac{2c\Gamma(\frac{p}{\delta}+\frac{b+2}{2})}{\Gamma(q+\frac{p}{\delta}+\frac{b+2}{2})} \text{ \ and \ } \kappa_{2}=\frac{4c^{2}\Gamma^{2}(\frac{p}{\delta}+\frac{b+2}{2})}{\Gamma^{2}(q+\frac{p}{\delta}+\frac{b+2}{2})}-\frac{3c^{2}\Gamma(\frac{p}{\delta}+\frac{b+2}{2})}{\Gamma(2q+\frac{p}{\delta}+\frac{b+2}{2})} \]
in particular, for $k=1$ from the above Euler-Rayleigh inequalities we have the next inequality for  $4{}_{q}\varsigma_{p,b,c,\delta,1}$, that is,
\[\frac{2\Gamma(q+\frac{p}{\delta}+\frac{b+2}{2})}{c\Gamma(\frac{p}{\delta}+\frac{b+2}{2})}<r^{\star}({}_{q}h_{p,b,c,\delta})<\frac{8\Gamma(q+\frac{p}{\delta}+\frac{b+2}{2})\Gamma(2q+\frac{p}{\delta}+\frac{b+2}{2})}{c\left\lbrace4\Gamma(\frac{p}{\delta}+\frac{b+2}{2})\Gamma(2q+\frac{p}{\delta}+\frac{b+2}{2})-3\Gamma^{2}(q+\frac{p}{\delta}+\frac{b+2}{2}) \right\rbrace }, \]
of course we can obtain more tighter bounds for other values of $k\in \mathbb{N}.$
\end{proof}
\subsection{The radii of convexity of order $\alpha$ of the functions ${}_{q}f_{p,b,c,\delta}$, ${}_{q}g_{p,b,c,\delta}$ and ${}_{q}h_{p,b,c,\delta}$.} In this section we aim to determine the radii of convexity of the normalized generalized Struve functions and to find tight lower and upper bounds for the radius of convexity of these normalized Struve functions with the help of Euler-Rayleigh inequalities. 

\begin{theorem}\label{MainTheo4}
Let $\delta,q,b,c>0$,  $p+1>0 \text{ \ and \ } \alpha \in\left[0,1 \right).$ Then the following assertions hold true.
\item[\bf a.] The radius of convexity of order $\alpha$ of the function ${}_{q}f_{p,b,c,\delta}$ is the smallest root of the equation
\[1+\left( \frac{1}{p+1}-1\right)r\frac{ {}_{q}W_{p,b,c,\delta}'(r)}{{}_{q}W_{p,b,c,\delta}(r)}+r\frac{ {}_{q}W_{p,b,c,\delta}''(r)}{{}_{q}W_{p,b,c,\delta}'(r)}=\alpha. \]
\item[\bf b.] The radius of convexity of order $\alpha$ of the function ${}_{q}g_{p,b,c,\delta}$ is the smallest root of the equation
\[1+r\frac{{}_{q}g_{p,b,c,\delta}''(r)}{{}_{q}g_{p,b,c,\delta}'(r)}=\alpha.\]

\item[\bf c.] The radius of convexity of order $\alpha$ of the function ${}_{q}h_{p,b,c,\delta}$ is the smallest root of the equation
\[1+z\frac{{}_{q}h_{p,b,c,\delta}''(z)}{{}_{q}h_{p,b,c,\delta}'(z)}=\alpha.\]
\end{theorem}
\begin{proof}
\item [\bf a.]
It is obvious that 
\[ 1+ z\frac{ {}_{q}f_{p,b,c,\delta}''(z)}{{}_{q}f_{p,b,c,\delta}'(z)}=1+\left( \frac{1}{p+1}-1\right)z\frac{ {}_{q}W_{p,b,c,\delta}'(z)}{{}_{q}W_{p,b,c,\delta}(z)}+z\frac{ {}_{q}W_{p,b,c,\delta}''(z)}{{}_{q}W_{p,b,c,\delta}'(z)}. \]
Now, by making use of \Cref{mainlemma} we can write the following infinite product representations
\begin{align*}
2^{p+1}\Gamma(\frac{p}{\delta}+\frac{b+2}{2}){}_{q}W_{p,b,c,\delta}(z)&=z^{p+1}\prod_{n\geq 1}\left(1-\frac{z^2}{{}_{q}\omega_{p,b,c,\delta,n}^2}\right), \\
2^{p+1}\Gamma(\frac{p}{\delta}+\frac{b+2}{2}){}_{q}W_{p,b,c,\delta}'(z)&=(p+1)z^{p}\prod_{n\geq 1}\left(1-\frac{z^2}{{}_{q}\omega_{p,b,c,\delta,n}'^2}\right)
\end{align*}
where ${}_{q}\omega_{p,b,c,\delta,n}$ and ${}_{q}\omega_{p,b,c,\delta,n}'$ stand for the $n$th positive roots of ${}_{q}W_{p,b,c,\delta}$ and ${}_{q}W_{p,b,c,\delta}',$ respectively. Logarithmic differentiation of the above equality leads to
\[ z\frac{{}_{q}W_{p,b,c,\delta}'(z)}{{}_{q}W_{p,b,c,\delta}(z)}= p+1-\sum_{n \geq 1} \frac{2z^2}{{}_{q}\omega_{p,b,c,\delta,n}^2-z^2}, \text{ \ \ } z\frac{{}_{q}W_{p,b,c,\delta}''(z)}{{}_{q}W_{p,b,c,\delta}'(z)}=p-\sum_{n \geq 1}\frac{2z^2}{{}_{q}\omega_{p,b,c,\delta,n}'^2-z^2},\]
which implies that 
\[1+z\frac{{}_{q}f_{p,b,c,\delta}''(z)}{{}_{q}f_{p,b,c,\delta}'(z)}=1-\left(\frac{1}{p+1}-1\right)\sum_{n\geq1}\frac{2z^2}{{}_{q}\omega_{p,b,c,\delta,n}^2-z^2}-\sum_{n \geq 1}\frac{2z^2}{{}_{q}\omega_{p,b,c,\delta,n}'^2-z^2}.\]
We will prove the theorem in two steps. First suppose $p\in \left(-1,0 \right] .$ By using the inequality \eqref{lemineq2}, for all $z\in \mathbb{D}(0,{}_{q}\omega_{p,b,c,n}')$ it is easy to deduce the following inequality
\[\real\left(1+z\frac{{}_{q}f_{p,b,c,\delta}''(z)}{{}_{q}f_{p,b,c,\delta}'(z)}\right)\geq1-\left(\frac{1}{p+1}-1\right)\sum_{n\geq1}\frac{2r^2}{{}_{q}\omega_{p,b,c,\delta,n}^2-r^2}-\sum_{n \geq 1}\frac{2r^2}{{}_{q}\omega_{p,b,c,\delta,n}'^2-r^2},\]
where $\left| z\right|=r. $ In the second step, it is easy to see that if we use the inequality \eqref{lemineq1} then we conclude that the above inequality
is also fulfilled when $p>0$ . Here we used that the zeros ${}_{q}\omega_{p,b,c,\delta,n}$ and ${}_{q}\omega_{p,b,c,\delta,n}'$ interlace according to \Cref{mainlemma}. Now, the above deduced inequality implies for $r\in\left(0, {}_{q}\omega_{p,b,c,\delta,1}'\right)$
\[\inf_{z\in\mathbb{D}_{r}}\left\lbrace\real\left(1+z\frac{{}_{q}f_{p,b,c,\delta}''(z)}{{}_{q}f_{p,b,c,\delta}'(z)}\right)\right\rbrace=1+r\frac{{}_{q}f_{p,b,c,\delta}''(r)}{{}_{q}f_{p,b,c,\delta}'(r)}.\]
Now we deal with the function ${}_{q}\vartheta_{p,b,c,\delta}:\left(0, {}_{q}\omega_{p,b,c,\delta,1}' \right) \rightarrow \mathbb{R},$ defined by
\[ {}_{q}\vartheta_{p,b,c,\delta}(r)=1+r\frac{{}_{q}f_{p,b,c,\delta}''(r)}{{}_{q}f_{p,b,c,\delta}'(r)}.\]
The function is strictly decreasing since
\begin{align*}
{}_{q}\vartheta_{p,b,c,\delta}'(r)&=-\left(\frac{1}{p+1}-1 \right)\sum_{n \geq 1} \frac{4r{}_{q}\omega_{p,b,c,\delta,n}^2}{\left({}_{q}\omega_{p,b,c,\delta,n}^2-r^2\right) ^2}-\sum_{n\geq1}\frac{4r{}_{q}\omega_{p,b,c,\delta,n}'^2}{\left({}_{q}\omega_{p,b,c,\delta,n}'^2-r^2\right)^2}\\
&<\sum_{n\geq1}\frac{4r{}_{q}\omega_{p,b,c,\delta,n}^2}{\left({}_{q}\omega_{p,b,c,\delta,n}^2-r^2\right)^2}-\sum_{n\geq1}\frac{4r{}_{q}\omega_{p,b,c,\delta,n}'^2}{\left({}_{q}\omega_{p,b,c,\delta,n}'^2-r^2\right)^2}<0
\end{align*}
for $p+1>0$ and $r\in\left(0, {}_{q}\omega_{p,b,c,\delta,1}' \right).$ It is important to note that here we used again that the zeros ${}_{q}\omega_{p,b,c,\delta,n}$ and ${}_{q}\omega_{p,b,c,\delta,n}'$ interlace and with conditions of \Cref{mainlemma} we get that
\[{}_{q}\omega_{p,b,c,\delta,n}^2\left( {}_{q}\omega_{p,b,c,\delta,1}'^2-r^2\right)^2<{}_{q}\omega_{p,b,c,\delta,n}'^2\left( {}_{q}\omega_{p,b,c,\delta,1}^2-r^2\right)^2.\]
Also, taking into consideration that $\lim_{r \searrow 0}{}_{q}\vartheta_{p,b,c,\delta}(r)=1-\alpha>0,$ $\lim_{r\nearrow {}_{q}\omega_{p,b,c,\delta,1}}{}_{q}\vartheta_{p,b,c,\delta}(r)=-\infty$ that means that for $z\in \mathbb{D}(0,r_1)$ we have
\[\real\left(1+z\frac{{}_{q}f_{p,b,c,\delta}''(z)}{{}_{q}f_{p,b,c,\delta}'(z)} \right) >\alpha,\]
if and only if $r_1$ is the unique root of
\[1+z\frac{{}_{q}f_{p,b,c,\delta}''(z)}{{}_{q}f_{p,b,c,\delta}'(z)}=\alpha,\]
situated in $\left(0, {}_{q}\omega_{p,b,c,\delta,1}' \right). $
\item [\bf b.] The definition of the function ${}_{q}g_{p,b,c,\delta}$ implies that this function is of the Laguerre-P\'olya class of entire functions. That is, the function ${}_{q}g_{p,b,c,\delta}$ is entire function that has only real zeros. Suppose that \({}_{q}\beta_{p,b,c,\delta,n}\)'s are the real zeros of the function ${}_{q}g_{p,b,c,\delta}$.  By making use of \eqref{mainlemma1} we conclude
\[1+z\frac{{}_{q}g_{p,b,c,\delta}''(z)}{{}_{q}g_{p,b,c,\delta}'(z)}=1-\sum_{n \geq 1}\frac{2z^2}{{}_{q}\beta_{p,b,c,\delta,n}^2-z^2}.\]
By applying the inequality \eqref{lemineq2} we obtain
\[\real\left(1+z\frac{{}_{q}g_{p,b,c,\delta}''(z)}{{}_{q}g_{p,b,c,\delta}'(z)} \right) \geq 1- \sum_{n \geq 1}\frac{2r^2}{{}_{q}\beta_{p,b,c,\delta,n}^2-r^2}, \]
where $\left|z \right|=r. $ Whence for $r\in\left(0,{}_{q}\beta_{p,b,c,\delta,1} \right) $ we have
\[ \inf_{z\in \mathbb{D}_{r}}\left\lbrace \real\left(1+z\frac{{}_{q}g_{p,b,c,\delta}''(z)}{{}_{q}g_{p,b,c,\delta}'(z)} \right)\right\rbrace= 1- \sum_{n \geq 1}\frac{2r^2}{{}_{q}\beta_{p,b,c,\delta,n}^2-r^2}=1+r\frac{{}_{q}g_{p,b,c,\delta}''(r)}{{}_{q}g_{p,b,c,\delta}'(r)}.\]
The function ${}_{q}\phi_{p,b,c,\delta}:\left(0,{}_{q}\beta_{p,b,c,\delta,1} \right)\rightarrow \mathbb{R}$ defined by
\[{}_{q}\phi_{p,b,c,\delta}(r)=1+r\frac{{}_{q}g_{p,b,c,\delta}''(r)}{{}_{q}g_{p,b,c,\delta}'(r)}, \]
is strictly decreasing and 
\[\lim_{r\nearrow {}_{q}\beta_{p,b,c,\delta,1}}{}_{q}\phi_{p,b,c,\delta}(r)=-\infty, \text{ \ \ } \lim_{r \searrow 0}{}_{q}\phi_{p,b,c,\delta}(r)=1. \]
As a result, the equation
\[ 1+r\frac{{}_{q}g_{p,b,c,\delta}''(r)}{{}_{q}g_{p,b,c,\delta}'(r)}=\alpha\]
has a unique root $r_2$ situated in $\left(0,{}_{q}\beta_{p,b,c,\delta,1} \right).$
\item[\bf c.] It is obvious that 
\[1+z\frac{{}_{q}h_{p,b,c,\delta}''(z)}{{}_{q}h_{p,b,c,\delta}'(z)}=1-\sum_{n \geq 1}\frac{z}{{}_{q}\theta_{p,b,c,\delta,n}^2-z}.\]
Let $r \in (0, {}_{q}\theta_{p,b,c,\delta,1}^2)$ be a fixed number. Because of the minimum principle for harmonic functions and inequality \eqref{lemineq1} for $\lambda=0$ we get
\[\real\left(1+z\frac{{}_{q}h_{p,b,c,\delta}''(z)}{{}_{q}h_{p,b,c,\delta}'(z)} \right)=1-\real\left(\sum_{n\geq1}\frac{z}{{}_{q}\theta_{p,b,c,\delta,n}^2-z}\right) \geq 1+r\frac{{}_{q}h_{p,b,c,\delta}''(r)}{{}_{q}h_{p,b,c,\delta}'(r)}. \]
Consequently, it follows that
\[\inf_{z\in \mathbb{D}_{r}}\left\lbrace \real\left( 1+z\frac{{}_{q}h_{p,b,c,\delta}''(z)}{{}_{q}h_{p,b,c,\delta}'(z)}\right) \right\rbrace=1+r\frac{{}_{q}h_{p,b,c,\delta}''(r)}{{}_{q}h_{p,b,c,\delta}'(r)}. \]
The function ${}_{q}\Theta_{p,b,c,\delta}:(0,{}_{q}\theta_{p,b,c,\delta,1}^2)\rightarrow \mathbb{R}$ defined by
\[{}_{q}\Theta_{p,b,c,\delta}(r)= 1+r\frac{{}_{q}h_{p,b,c,\delta}''(r)}{{}_{q}h_{p,b,c,\delta}'(r)}\]
is strictly decreasing and 
\[\lim_{r\nearrow {}_{q}\theta_{p,b,c,\delta,1}^2}{}_{q}\Theta_{p,b,c,\delta}(r)=-\infty, \text{ \ \ } \lim_{r \searrow 0}{}_{q}\Theta_{p,b,c,\delta}(r)=1.\]
Consequently, the equation
\[ 1+z\frac{{}_{q}h_{p,b,c,\delta}''(z)}{{}_{q}h_{p,b,c,\delta}'(z)}=\alpha\]
has a unique root $r_3$ in $(0,{}_{q}\theta_{p,b,c,\delta,1}^2).$

The proof is completed.
\end{proof}
The following theorems are related to some tight lower and upper bounds for the radii of convexity of the normalized generalized Struve functions.
\begin{theorem}\label{MainTheorem5}
Let $\delta,q,b,c>0$ and $p+1>0$. Then the radius of convexity $r^{c}({}_{q}g_{p,b,c,\delta})$ of the function 
\[z\mapsto {}_{q}g_{p,b,c,\delta}(z)=2^{p+1} \Gamma(\frac{p}{\delta}+\frac{b+2}{2})z^{-p}{}_{q}W_{p,b,c,\delta}(z),
 \]
 is the smallest root of the $\left(z{}_{q}g_{p,b,c,\delta}'\right)^{\prime}=0$ and satisfies the following inequality
\[\frac{2}{3}\sqrt{\frac{\Gamma(q+\frac{p}{\delta}+\frac{b+2}{2})}{c\Gamma(\frac{p}{\delta}+\frac{b+2}{2})}}<r^{c}({}_{q}g_{p,b,c,\delta})<6\sqrt{\frac{\Gamma(q+\frac{p}{\delta}+\frac{b+2}{2})\Gamma(2q+\frac{p}{\delta}+\frac{b+2}{2})}{c\left\lbrace\Gamma(\frac{p}{\delta}+\frac{b+2}{2})\Gamma(2q+\frac{p}{\delta}+\frac{b+2}{2})-25\Gamma^{2}(\frac{p}{q+\delta}+\frac{b+2}{2}) \right\rbrace }}.\]
\end{theorem}
\begin{proof}
For proving our main result we will need the Alexander’s duality theorem which has a very simple proof based on the characterization of starlike and convex functions in the unit disc. Owing to this theorem one can deduce that the function ${}_{q}g_{p,b,c,\delta}(z)$ is convex if and only if $z\mapsto \left( z{}_{q}g_{p,b,c,\delta}\right)^{\prime}$ is starlike. From the studies in \cite{BKS,BOS} we know that the smallest positive zero of $z\mapsto\left(z{}_{q}g_{p,b,c,\delta}'\right)^{\prime}$ is the radius of starlikeness of $z{}_{q}g_{p,b,c,\delta}'(z).$ That is why the radius of convexity $r^{c}({}_{q}g_{p,b,c,\delta})$ is the smallest positive root of the equation $\left(z{}_{q}g_{p,b,c,\delta}'\right)^{\prime}=0.$ Now, by using the infinite series representations of the generalized Struve fuction and its derivative
\begin{align}
{}_{q}W_{p,b,c,\delta}(z)&=\sum_{n\geq 0}\frac{(-1)^{n}c^nz^{2n+p+1}}{2^{2n+p+1}n!\Gamma(qn+\frac{p}{\delta}+\frac{b+2}{2})}, \label{MainTheo5-1}\\
{}_{q}W_{p,b,c,\delta}'(z)&=\sum_{n\geq 0}\frac{(-1)^{n}c^n(2n+p+1)z^{2n+p}}{2^{2n+p+1}n!\Gamma(qn+\frac{p}{\delta}+\frac{b+2}{2})} \label{MainTheo5-2},
\end{align}
respectively, we arrive at 
\[{}_{q}\Delta_{p,b,c,\delta}(z)=\left(z{}_{q}g_{p,b,c,\delta}'(z) \right)'=\Gamma(\tfrac{p}{\delta}+\tfrac{b+2}{2})\sum_{n\geq 0}\frac{(-1)^{n}c^n(2n+1)^{2}}{2^{2n}n!\Gamma(qn+\frac{p}{\delta}+\frac{b+2}{2})}z^{2n}.\]
Moreover, it is obvious that
\begin{equation}\label{Lambda1}
{}_{q}\Lambda_{p,b,c,\delta}(z)={}_{q}\Delta_{p,b,c,\delta}(2\sqrt{z})=\Gamma(\tfrac{p}{\delta}+\tfrac{b+2}{2})\sum_{n\geq 0}\frac{(-1)^{n}c^n(2n+1)^{2}}{n!\Gamma(qn+\frac{p}{\delta}+\frac{b+2}{2})}z^{n}.
\end{equation}
Taking into account facts that the function ${}_{q}g_{p,b,c,\delta}$ belongs to the Laguerre-P\'olya class of entire functions and that the class $\mathcal{LP}$ is closed under differentiation, it is easy to deduce that the function ${}_{q}\Lambda_{p,b,c,\delta}$ belongs to the Laguerre-P\'olya class. As a result, the function ${}_{q}\Lambda_{p,b,c,\delta}$ is an entire function that has only real zeros. In this case, we can continue to prove the theorem by emulating the steps taken in the proof of \Cref{MainTheo2}.  Suppose that ${}_{q}\varrho_{p,b,c,\delta,n}$'s are the positive zeros of the function ${}_{q}\Lambda_{p,b,c,\delta}.$ Then the function ${}_{q}\Lambda_{p,b,c,\delta}$ has the infinite product representation as follows:
\begin{equation}\label{Lambda2}
{}_{q}\Lambda_{p,b,c,\delta}(z)=\prod_{n\geq 1}\left( 1-\frac{z}{{}_{q}\varrho_{p,b,c,\delta,n}}\right). 
\end{equation}
Having taken the logarithmic derivation of the above equality, we obtain
\begin{equation} \label{Lambda3}
\frac{{}_{q}\Lambda_{p,b,c,\delta}'(z)}{{}_{q}\Lambda_{p,b,c,\delta}(z)}=-\sum_{n\geq 1}\frac{1}{{}_{q}\varrho_{p,b,c,\delta,n}-z}=-\sum_{k\geq0}\mu_{k+1}z^k,\text{ \ \ } \left|z \right|< {}_{q}\varrho_{p,b,c,\delta,1},
\end{equation}
where $\mu_{k}=\sum_{n\geq 1}{}_{q}\varrho_{p,b,c,\delta,n}^{-k}.$ Furthermore, by making use of the function \( {}_{q}\Lambda_{p,b,c,\delta}\) we obtain
\begin{equation}\label{Lambda4}
\frac{{}_{q}\Lambda_{p,b,c,\delta}'(z)}{{}_{q}\Lambda_{p,b,c,\delta}(z)}=\sum_{n\geq 0}\frac{(-1)^{n}c^n(2n+1)^{2}}{n!\Gamma(qn+\frac{p}{\delta}+\frac{b+2}{2})}z^{n} \bigg/ \sum_{n\geq 0}\frac{(-1)^{n+1}c^{n+1}(2n+3)^{2}}{n!\Gamma(q{n+1}+\frac{p}{\delta}+\frac{b+2}{2})}z^{n}.
\end{equation}
It is evident that equations \eqref{Lambda3} and \eqref{Lambda4} enable us to express the Euler-Rayleigh sums in terms of $q,p,b,c,\delta$ and by considering the  Euler-Rayleigh inequalities \(\mu_{k}^{-\frac{1}{k}}<{}_{q}\varrho_{p,b,c,\delta,1}<\frac{\mu_{k}}{\mu_{k+1}} \) we obtain the inequalities for $2\sqrt{{}_{q}\varrho_{p,b,c,\delta,1}}$  for $\delta,q,b,c>0$, $p+1>0$ and $k\in \mathbb{N}$
\[2\sqrt{\mu_{k}^{-\frac{1}{k}}}<r^{c}({}_{q}g_{p,b,c,\delta})<2\sqrt{\frac{\mu_{k}}{\mu_{k+1}}}.\]
As a result of the comparison of the coefficients of \eqref{Lambda3} and \eqref{Lambda4}, we arrive at
\[\mu_{1}=\frac{9c\Gamma(\tfrac{p}{\delta}+\tfrac{b+2}{2})}{\Gamma(q+\tfrac{p}{\delta}+\tfrac{b+2}{2})} \text{ \ and \ } \mu_{2}=\frac{81c^2\Gamma^{2}(\tfrac{p}{\delta}+\tfrac{b+2}{2})}{\Gamma^{2}(q+\tfrac{p}{\delta}+\tfrac{b+2}{2})}-\frac{25c^2\Gamma(\tfrac{p}{\delta}+\tfrac{b+2}{2})}{\Gamma(2q+\tfrac{p}{\delta}+\tfrac{b+2}{2})}. \]
Obviously, for $k=1$ from the Euler-Rayleigh inequalities, It can be observed that the following inequality come to light
\[\frac{2}{3}\sqrt{\frac{\Gamma(q+\frac{p}{\delta}+\frac{b+2}{2})}{c\Gamma(\frac{p}{\delta}+\frac{b+2}{2})}}<r^{c}({}_{q}g_{p,b,c,\delta})<6\sqrt{\frac{\Gamma(q+\frac{p}{\delta}+\frac{b+2}{2})\Gamma(2q+\frac{p}{\delta}+\frac{b+2}{2})}{c\left\lbrace\Gamma(\frac{p}{\delta}+\frac{b+2}{2})\Gamma(2q+\frac{p}{\delta}+\frac{b+2}{2})-25\Gamma^{2}(\frac{p}{q+\delta}+\frac{b+2}{2}) \right\rbrace }}. \]
No doubt it can be presented more tighter bounds for other values $k\in \mathbb{N}.$
\end{proof}	

\begin{theorem}\label{MainTheo6}
Let $\delta,q,b,c>0$ and $p+1>0$. Then the radius of convexity $r^{c}({}_{q}h_{p,b,c,\delta})$ of the function 
\[z\mapsto{}_{q}h_{p,b,c,\delta}(z)= 2^{p+1}\Gamma(\frac{p}{\delta}+\frac{b+2}{2})z^{1-\frac{p+1}{2}}{}_{q}W_{p,b,c,\delta}(\sqrt{z})\]
is the smallest root of the $\left(z{}_{q}h_{p,b,c,\delta}'(z)\right)^{\prime}=0$ and satisfies the following inequality
\[\frac{\Gamma(q+\tfrac{p}{\delta}+\tfrac{b+2}{2})}{c\Gamma(\tfrac{p}{\delta}+\tfrac{b+2}{2})}<r^{c}({}_{q}h_{p,b,c,\delta})<\frac{16\Gamma(q+\tfrac{p}{\delta}+\tfrac{b+2}{2})\Gamma(2q+\tfrac{p}{\delta}+\tfrac{b+2}{2})}{c\left\lbrace16\Gamma(\tfrac{p}{\delta}+\tfrac{b+2}{2})\Gamma(2q+\tfrac{p}{\delta}+\tfrac{b+2}{2})-9\Gamma^{2}(q+\tfrac{p}{\delta}+\tfrac{b+2}{2}) \right\rbrace }.\]
\end{theorem}
\begin{proof}
In light of explanations which have been presented in the proof of previous theorem (that is, \autoref{MainTheorem5}) one can deduce that the radius of convexity $r^{c}({}_{q}h_{p,b,c,\delta})$ is the smallest positive root of the equation $\left(z{}_{q}h_{p,b,c,\delta}'(z)\right)^{\prime}=0.$ By taking $\sqrt{z}$ instead of $z$ in \eqref{MainTheo5-1} and \eqref{MainTheo5-2}, respectively we can draw a conclusion
\begin{equation}\label{MainTheo6-1}
{}_{q}\lambda_{p,b,c,\delta}(z)=\left(z{}_{q}h_{p,b,c,\delta}'(z)\right)^{\prime}=1+\Gamma(\tfrac{p}{\delta}+\tfrac{b+2}{2})\sum_{n\geq 1}\frac{(-1)^{n}c^n(n+1)^{2}}{2^{2n}n!\Gamma(qn+\frac{p}{\delta}+\frac{b+2}{2})}z^{n}.
\end{equation}
With the help of facts that the function ${}_{q}h_{p,b,c,\delta}$ belongs to the Laguerre-P\'olya class of entire functions $\mathcal{LP}$ and that the class  $\mathcal{LP}$ is closed under differentiation, it is not hard to say that the function ${}_{q}\lambda_{p,b,c,\delta}$ is also in the Laguerre-P\'olya class. This means that the zeros of the function ${}_{q}\lambda_{p,b,c,\delta}$ are all real. Furthermore, by means of approach used in the proof of \autoref{MainTheo2} we can say that all of its zeros are positive.  Suppose that  ${}_{q}\tau_{p,b,c,\delta,n}$s are the zeros of the function ${}_{q}\lambda_{p,b,c,\delta}.$ Then the infinite product representation of the function ${}_{q}\lambda_{p,b,c,\delta}$ can be given as
\begin{equation}\label{MainTheo6-2}
{}_{q}\lambda_{p,b,c,\delta}(z)=\prod_{n\geq 1}\left(1-\frac{z}{{}_{q}\tau_{p,b,c,\delta,n}} \right).
\end{equation}
Having taken the logarithmic derivation of Eq. \eqref{MainTheo6-2}, we get
\begin{equation}\label{MainTheo6-3}
\frac{{}_{q}\lambda_{p,b,c,\delta}'(z)}{{}_{q}\lambda_{p,b,c,\delta}(z)}=-\sum_{k\geq 0}\upsilon_{k+1}z^{k}, \text{ \ \ } \left|z \right|<{}_{q}\tau_{p,b,c,\delta,1},
\end{equation}
where $\upsilon_{k}=\sum_{n \geq 1} {}_{q}\tau_{p,b,c,\delta,n}^{-k}.$ Also, by making use of the derivative of Eq. \eqref{MainTheo6-1}, which is related to infinite sum representation of ${}_{q}\lambda_{p,b,c,\delta}(z)$, we obtain
\begin{equation}\label{MainTheo6-4}
\frac{{}_{q}\lambda_{p,b,c,\delta}'(z)}{{}_{q}\lambda_{p,b,c,\delta}(z)}=\sum_{n\geq0}\frac{(-1)^{n+1}c^{n+1}(n+2)^{2}}{2^{2n+2}n!\Gamma(q(n+1)+\frac{p}{\delta}+\frac{b+2}{2})}z^{n}  \bigg/ \sum_{n\geq 0} \frac{(-1)^{n}c^n(n+1)^{2}}{2^{2n}n!\Gamma(qn+\frac{p}{\delta}+\frac{b+2}{2})}z^{n}.
\end{equation}
The outcomes found by comparing the coefficients of Eq. \eqref{MainTheo6-3} and \eqref{MainTheo6-4} can be given as follows:
\[\upsilon_{1}=\frac{c\Gamma(\tfrac{p}{\delta}+\tfrac{b+2}{2})}{\Gamma(q+\tfrac{p}{\delta}+\tfrac{b+2}{2})}, \text{ \ \ } \upsilon_{2}=\frac{c^{2}\Gamma^{2}(\tfrac{p}{\delta}+\tfrac{b+2}{2})}{\Gamma^{2}(q+\tfrac{p}{\delta}+\tfrac{b+2}{2})}-\frac{9c^{2}\Gamma(\tfrac{p}{\delta}+\tfrac{b+2}{2})}{16\Gamma(2q+\tfrac{p}{\delta}+\tfrac{b+2}{2})}.\]
By considering the Euler-Rayleigh inequalities $\upsilon_{k}^{-\frac{1}{k}}<{}_{q}\tau_{p,b,c,\delta,1}<\frac{\upsilon_{k}}{\upsilon_{k+1}}$ for $\delta,q,b,c>0$, $p+1>0$ and $k=1$ we have
\[\frac{\Gamma(q+\tfrac{p}{\delta}+\tfrac{b+2}{2})}{c\Gamma(\tfrac{p}{\delta}+\tfrac{b+2}{2})}<r^{c}({}_{q}h_{p,b,c,\delta})<\frac{16\Gamma(q+\tfrac{p}{\delta}+\tfrac{b+2}{2})\Gamma(2q+\tfrac{p}{\delta}+\tfrac{b+2}{2})}{c\left\lbrace16\Gamma(\tfrac{p}{\delta}+\tfrac{b+2}{2})\Gamma(2q+\tfrac{p}{\delta}+\tfrac{b+2}{2})-9\Gamma^{2}(q+\tfrac{p}{\delta}+\tfrac{b+2}{2}) \right\rbrace }.\]
Obviously, it can be presented more tighter bounds for other values $k\in \mathbb{N}.$
\end{proof}
\subsection{Some particular cases of the main results}
This section is devoted to the some intriguing results which are obtained by making some comparisons with earlier results. It is possible to say that generalized Struve function is actually a generalization of the suitible transformation of the Bessel function of the first kind. That is, we have the relation
\[{}_{1}W_{\nu-1,2,1,1}(z)=J_{\nu}(z) \]
where \(J_{\nu}\) is the Bessel function of the first kind and order \(\nu.\) By considering this relation, we see that our main results obtained in this paper coincide with the following listed studies. 

It is evident that our main results which are given in \autoref{MainTheo1}, in particular for \(q=2, p=\nu-1, b=2, c=1 \text{ \and\ } \delta=1,\) correspond to the results in \cite[Theorem 1]{BKS}.
\begin{corollary}
Let \(\nu>-1 \text{ \ and \ } \alpha\in\left[0,1 \right). \)
\begin{itemize}
	\item [a.] The radius of starlikeness of order \(\alpha\) of \({}_{1}f_{\nu-1,2,1,1}(z)=\left[2^{\nu}\Gamma(\nu+1)J_{\nu}(z)\right]^{\frac{1}{\nu}} \) is the smallest positive root of the equation
	\[zJ_{\nu}'(z)-\alpha\nu J_{\nu}(z)=0.\]
	\item [b.] The radius of starlikeness of order \(\alpha\) of \({}_{1}g_{\nu-1,2,1,1}(z)=2^{\nu}\Gamma(\nu+1)z^{1-\nu}J_{\nu}(z)\) is the smallest positive root of the equation
	\[zJ_{\nu}'(z)+(1-\alpha-\nu)J_{\nu}(z)=0.\]
	\item [c.] The radius of starlikeness of order \(\alpha\) of \({}_{1}h_{\nu-1,2,1,1}(z)=2^{\nu}\Gamma(\nu+1)z^{1-\frac{\nu}{2}}J_{\nu}(\sqrt{z})\) is the smallest positive root of the equation
	\[zJ_{\nu}'(z)+(2-2\alpha-\nu)J_{\nu}(z)=0.\]
\end{itemize} 
\end{corollary}

By choosing the values \(q=2, p=\nu-1, b=2, c=1 \text{ \and\ } \delta=1\) in \autoref{MainTheo1'} we obtain the following corollary which is new.
\begin{corollary}
If \(\nu>-1\), then the radius of starlikeness \(r^{\star}({}_{1}f_{\nu-1,2,1,1})\) satisfies
\[2\sqrt{\frac{\nu(\nu+1)}{\nu+2}}<r^{\star}({}_{1}f_{\nu-1,2,1,1})<2(\nu+2)\sqrt{\frac{\nu(\nu+1)}{\nu^2+8\nu+8}}.\]
\end{corollary}

It is obvious that our main results which are presented in \autoref{MainTheo2} and \autoref{MainTheo3} when we take \(q=2, p=\nu-1, b=2, c=1 \text{ \and\ } \delta=1,\) coincide with the inequalities in \cite[Thm. 1]{ABY} and \cite[Thm. 2]{ABY}, respectively.
\begin{corollary}
Let \(\nu>-1 \). The following assertions hold true:
\begin{itemize}
	\item [a.] The radius of starlikeness \(r^{\star}({}_{1}g_{\nu-1,2,1,1})\) satisfies
	\[2\sqrt{\frac{\nu+1}{3}}<r^{\star}({}_{1}g_{\nu-1,2,1,1})<2\sqrt{\frac{3(\nu+1)(\nu+2)}{4\nu+13}} .\]
	\item [b.] The radius of starlikeness \(r^{\star}({}_{1}h_{\nu-1,2,1,1})\) satisfies
	\[2(\nu+1)<r^{\star}({}_{1}h_{\nu-1,2,1,1})<\frac{8(\nu+1)(\nu+2)}{\nu+5} .\]
\end{itemize}
\end{corollary}

It is clear that  our main results which are given in  \autoref{MainTheo4} when we choose \(q=2, p=\nu-1, b=2, c=1 \text{ \and\ } \delta=1,\) correspond to the results given in \cite[Thm. 1]{BSz2}, \cite[Thm. 2]{BSz2} and \cite[Thm. 3]{BSz2}.
\begin{corollary}
Let \(\nu>-1 \text{ \ and \ } \alpha\in\left[0,1 \right). \) Then the following assertions hold true
\begin{itemize}
	\item [a.] The radius of convexity of order \(\alpha \) of the function \({}_{1}f_{\nu-1,2,1,1}\) is the smallest positive root of the equation
	\[1+\frac{rJ_{\nu}''(r)}{J_{\nu}'(r)}+\left(\frac{1}{\nu}-1 \right)\frac{rJ_{\nu}'(r)}{J_{\nu}(r)}=\alpha. \]
	\item [b.] The radius of convexity of order \(\alpha \) of the function \({}_{1}g_{\nu-1,2,1,1}\) is the smallest positive root of the equation
	\[1+\frac{rJ_{\nu+2}(r)-3J_{\nu+1}(r)}{J_{\nu}(r)-rJ_{\nu+1}(r)}=\alpha. \]
	Here we used the recurrence formula \(zJ_{\nu}'(z)=\nu J_{\nu}(z)-zJ_{\nu+1}(z).\)
	\item [c.] The radius of convexity of order \(\alpha \) of the function \({}_{1}h_{\nu-1,2,1,1}\) is the smallest positive root of the equation
	\[ 1+\frac{r^{\tfrac{1}{2}}}{2}.\frac{r^{\tfrac{1}{2}}J_{\nu+2}(r^{\tfrac{1}{2}})-4J_{\nu+1}(r^{\tfrac{1}{2}})}{2J_{\nu}(r^{\tfrac{1}{2}})-r^{\tfrac{1}{2}}J_{\nu+1}(r^{\tfrac{1}{2}})}=\alpha. \]
\end{itemize}
\end{corollary}

It is clear that  our main results which are given in \autoref{MainTheorem5} and \autoref{MainTheo6} when we choose \(q=2, p=\nu-1, b=2, c=1 \text{ \and\ } \delta=1,\) tally with the results given in \cite[Thm. 6]{ABO} and \cite[Thm. 7]{ABO}, respectively.
\begin{corollary}
Let \(\nu>-1 \). The following assertions hold true:
\begin{itemize}
	\item [a.] The radius of convexity \(r^{c}({}_{1}g_{\nu-1,2,1,1}) \) satisfies the following inequality:
	\[\frac{2\sqrt{\nu+1}}{3}<r^{c}({}_{1}g_{\nu-1,2,1,1})<6\sqrt{\frac{(\nu+1)(\nu+2)}{56\nu+137}}. \]
	\item [b.] The radius of convexity \(r^{c}({}_{1}h_{\nu-1,2,1,1}) \) satisfies the following inequality:
	\[ \nu+1 < r^{c}({}_{1}h_{\nu-1,2,1,1})<\frac{16(\nu+1)(\nu+2)}{7\nu+23}.\]
\end{itemize}
\end{corollary}


\begin{thebibliography}{width}
\bibitem{AB}
\textsc{I. Akta\c{s}, \'A. Baricz}, Bounds for radii of starlikeness of some $q-$Bessel functions, {\em Results in Mathematics,} 72(1--2) (2017): 947--963.  
	
\bibitem{ABO}
\textsc{I. Akta\c{s}, \'A. Baricz, H. Orhan}, Bounds for radii of starlikeness and convexity of some special functions, {\em Turk J Math} (2018) 42:211--226, doi:10.3906/mat-1610-41.

\bibitem{ABY}
\textsc{I. Akta\c{s}, \'A. Baricz, N. Ya\u{g}mur}, Bounds for the radii of
univalence of some special functions,{\em Math. Inequal. Appl.} 20(3) (2017), 825--843.

\bibitem{ATO}
\textsc{I. Akta\c{s}, E. Toklu, H. Orhan}, Radii of uniform convexity of some special functions, {\em Turk J Math} (2018) 42: 3010 – 3024, doi:10.3906/mat-1806-43.

\bibitem{Baricz}
\textsc{\'A. Baricz}, Generalized Bessel function of first kind, {\em Lecture Notes in Mathematics}, Springer, Berlin (2010).

\bibitem{BDOY}
\textsc{\'A. Baricz, D.K. Dimitrov, H. Orhan, N. Ya\u{g}mur},
Radii of starlikeness of some special functions, {\em Proc. Amer. Math. Soc.} 144(2016), 3355--3367, doi:10.1090/proc/13120.

\bibitem{BKS}
\textsc{\'{A}. Baricz, P.A. Kup\'{a}n, R. Sz\'{a}sz}, The radius of
starlikeness of normalized Bessel functions of the first kind, {\em Proc. Amer.
Math. Soc.} 142(5) (2014) 2019--2025.

\bibitem{BOS}
\textsc{\'{A}. Baricz, H. Orhan, R. Sz\'{a}sz}, The radius of $\alpha-$
convexity of normalized Bessel functions of the first kind, {\em Comput. Methods Funct. Theory} 16(1) (2016) 93--103.

\bibitem{BSan}
\textsc{\'{A}. Baricz, S. Sanjeev}, Zeros of some special entire functions, {\em Proc. Amer. Math. Soc.} 146(5) (2018) 2207--2216.

\bibitem{BSz1}
\textsc{\'{A}. Baricz, R. Sz\'{a}sz}, The radius of convexity of
normalized Bessel functions, {\em Anal. Math.} 41(3) (2015) 141--151.

\bibitem{BSz2}
\textsc{\'{A}. Baricz, R. Sz\'{a}sz}, The radius of convexity of
normalized Bessel functions of the first kind, {\em Anal. Appl.}(Singap.) 12(5) (2014)
485--509.

\bibitem{BTK}
\textsc{\'{A}. Baricz, E. Toklu, E. Kadıo\u{g}lu}, Radii of starlikeness and convexity of Wright functions. {\em Math. Commun.} 23(2018); 97--117.

\bibitem{BR}
\textsc{N. Bohra, V. Ravichandran}, Radii problems for normalized Bessel functions of the first kind {\em Comput. Methods Funct. Theory,} 18 (2018), 99--123.

\bibitem{Brown}
\textsc{R. K. Brown}, Univalence of Bessel functions. {\em Proc. Amer. Math. Soc.} 11(1960), 278--283.

\bibitem{DSz}
\textsc{E. Deniz, R. Sz\'{a}sz}, The radius of uniform convexity of Bessel functions, {\em J. Math. Anal.} 453(1) (2017) 572--588.

\bibitem{DC}
\textsc{D.K. Dimitrov, Y.B. Cheikh}, Laguerre polynomials as Jensen polynomials of Laguerre-P\'olya entire functions, {\em J. Comput. Appl. Math.} 233 (2009) 703--707.

\bibitem{Dur}
\textsc{P.L. Duren}, {\em Univalent Functions}, Grundlehren Math. Wiss. 259,
Springer, New York, 1983.

\bibitem{KT}
\textsc{E. Kreyszig, J.Todd}, The radius of univalence of Bessel functions. {\em Illinois J. Math.} 4(1960), 143--149.

\bibitem{Pathak}
\textsc{R.S. Pathak}, Certain convergence theorems and asymptotic properties of a generalization of Lommel and Maitlaind transformation, {\em Proc. Nat. Acad Sci. India}, 36(1), 81-86, (1966).

\bibitem{Runckel}
\textsc{Hans-J. Runckel}, \textit{Zeros of entire functions}, {\em Trans. Amer. Math. Soc.}, 143(1969), 343--362, DOI 10.2307/1995253. MR0252641

\bibitem{Wilf} 
\textsc{H.S. Wilf}, The radius of univalence of certain entire functions, {\em Illinois J. Math.} (1962) 242--244.

\bibitem{Wat}
\textsc{G.N. Watson}, {\em A Treatise of the Theory of Bessel Functions}, Cambridge Univ. Press, Cambridge, 1944

\end{thebibliography}
\end{document}